\documentclass{amsart}

\usepackage{amsfonts}
\usepackage{amsmath}
\usepackage{hyperref}

\newtheorem{theorem}{Theorem}
\theoremstyle{plain}

\newtheorem{corollary}{Corollary}

\newtheorem{proposition}{Proposition}
\newtheorem{remark}{Remark}

\numberwithin{equation}{section}

\begin{document}

\title[A Generalization of Hermite--Hadamard's Inequality] {A Generalization of Hermite--Hadamard's
Inequality }

\author[M.W. Alomari]{Mohammad W. Alomari}

\address{Department of Mathematics, Faculty of Science and
Information Technology, Irbid National University, 2600 Irbid
21110, Jordan.} \email{mwomath@gmail.com}

\date{\today}


\begin{abstract}
In literature the Hermite--Hadamard inequality was eligible for
many reasons, one of the most surprising and interesting that the
Hermite--Hadamard inequality combine the midpoint and trapezoid
formulae in an inequality. In this work, a Hermite-Hadamard like
inequality that combines the composite trapezoid and composite
midpoint formulae is proved. So that, the classical
Hermite--Hadamard inequality becomes a special case of the
presented result. Some Ostrowski's type inequalities for convex
functions are proved as well.
\end{abstract}

\maketitle

\section{Introduction}
Let $f:[a,b]\rightarrow \mathbb{R}$, be a twice
differentiable mapping such that $f^{\prime \prime }\left( x\right) $ exists on $%
(a,b)$ and $\left\| {f''} \right\|_\infty=\sup_{x\in \left(
a,b\right) }\left\vert {f^{\prime \prime }\left( x\right)
}\right\vert <\infty $. Then the midpoint inequality is known as:
\begin{align}
\label{midineq}\left| {\int_a^b {f\left( x \right)dx}  - \left( {b
- a} \right)f\left( {\frac{{a + b}}{2}} \right)} \right| \le
\frac{{\left( {b - a} \right)^3 }}{{24}}\left\| {f''}
\right\|_\infty,
\end{align}
and, the trapezoid inequality
\begin{align}
\label{trapineq} \left| {\int_a^b {f\left( x \right)dx}  - \left(
{b - a} \right)\frac{{f\left( a \right) + f\left( b \right)}}{2}}
\right| \le \frac{{\left( {b - a} \right)^3 }}{{12}}\left\| {f''}
\right\|_\infty,
\end{align}
also hold. Therefore, the integral $\int_a^b {f\left( x \right)dx}
$ can be approximated  in terms of the midpoint and the
trapezoidal rules, respectively such as:
\begin{eqnarray*}
\int_a^b {f\left( x \right)dx}  \cong \left( {b - a}
\right)f\left( {\frac{{a + b}}{2}} \right),
\end{eqnarray*}
and
\begin{eqnarray*}
\int_a^b {f\left( x \right)dx}  \cong \left( {b - a}
\right)\frac{{f\left( a \right) + f\left( b \right)}}{2},
\end{eqnarray*}
which are  combined in a useful and famous relationship, known as
the Hermite-Hadamard's inequality. That is,
\begin{eqnarray}
\label{eq1.3}f\left( {\frac{{a + b}}{2}} \right) \le \frac{1}{{b -
a}}\int\limits_a^b {f\left( x \right)dx}  \le \frac{{f\left( a
\right) + f\left( b \right)}}{2},
\end{eqnarray}
which hold for all convex functions $f$ defined on a real interval
$[a,b]$.

The real beginning was (almost) in the last twenty five years,
where, in 1992 Dragomir \cite{D5} published his article about
(\ref{eq1.3}). The main result in \cite{D5} was
\begin{theorem}
\label{dragomir.thm}Let $ f:\left[ {a,b} \right]$ is convex
function one can define the following mapping on $\left[ {0 ,1 }
\right]$ such as:
\begin{eqnarray*}
H\left( {t} \right) = \frac{1}{{\left( {b - a}
\right)}}\int\limits_a^b {f\left( {tx + \left( {1 - t}
\right)\frac{{a + b}}{2}} \right) dx},
\end{eqnarray*}
then,
\begin{enumerate}
\item $H$ is convex and monotonic non-decreasing on $[0,1]$.

\item One has the bounds for $H$
\begin{eqnarray*}
\mathop {\sup }\limits_{t \in \left[ {0,1} \right] } H\left( {t}
\right) = \frac{1}{{\left( {b - a} \right)}}\int\limits_a^b
{f\left( {x} \right)dx}  = H\left( {1} \right),
\end{eqnarray*}
and
\begin{eqnarray*}
\mathop {\inf }\limits_{t \in \left[ {0,1} \right]} H\left( {t}
\right) = f\left( {\frac{{a + b}}{2}} \right) = H\left( {0}
\right).
\end{eqnarray*}
\end{enumerate}
\end{theorem}
A few years after 1992, many authors have took (a real) attention
to the Hermite--Hadamard inequality and sequence of several works
under various assumptions for the function involved such as
bounded variation, convex, differentiable functions whose
$n$-derivative(s) belong to $L_p[a,b]$; ($1\le p\le \infty$),
Lipschitz, monotonic, ... etc, have been published. For a
comprehensive list of results and excellent bibliography we
recommend the interested to refer to \cite{C1},\cite{C2} and
\cite{D9}.

In 1997, Yang and Hong \cite{Y}, continued on Dragomir result
\ref{dragomir.thm} and they proved the following theorem:
\begin{theorem}
\label{yang.thm}Suppose that $ f:[a,b]\to \mathbb{R} $, is convex
and the mapping $  F:\left[ {0,1} \right] \to \mathbb{R}$ is
defined by
\begin{eqnarray*}
F\left( {t} \right)= \frac{1}{{b-a}}\int_{a}^{b} \left[ {f\left(
{\frac{{1 + t}}{2}a + \frac{{1 - t}}{2}u} \right)+ f\left(
{\frac{{1 + t}}{2}b + \frac{{1 - t}}{2}u} \right)} \right]du,
\end{eqnarray*}
then,
\begin{enumerate}
\item The mapping $F$ is convex and monotonic nondecreasing on $
\left[ {0,1} \right]$.

\item We have the bounds
\begin{eqnarray*}
\mathop {\inf }\limits_{t \in \left[ {0,1} \right] } F\left( {t,s}
\right) = \frac{1}{{\left( {b - a} \right)}}\int\limits_a^b
{f\left( {x}\right)dx}   = F\left( {0} \right)
\\
\mathop {\sup }\limits_{t \in \left[ {0,1} \right] } F\left( {t}
\right) = \frac{{f\left( {a} \right) + f\left( {b} \right)}}{2} =
F\left( {1} \right).
\end{eqnarray*}
\end{enumerate}
\end{theorem}
For other closely related results see
\cite{Akkouchi},\cite{alomari},\cite{D1}--\cite{D4} and
\cite{D6}--\cite{D8}.\\

In terms of composite numerical integration, we recall that
\emph{the Composite Midpoint rule} (\cite{B}, p.202)
\begin{align}
\label{Comp.Mid}\int_a^b {f\left( x \right)dx}  - 2h\sum\limits_{j
= 0}^{n/2} {f\left( {x_{2j} } \right)}= \frac{{\left( {b - a}
\right) }}{{6 }}h^2f''\left( {\mu} \right),
\end{align}
for some $\mu \in (a,b)$, where $f \in C^2[a,b]$, $n$ even, $h =
\frac{b-a}{n+2}$ and $x_j = a+(j+1)h$, for each
$j=-1,0,\cdots,n+1$.

And, \emph{the Composite Trapezoid rule} (\cite{B}, p.203)
\begin{align}
\label{Comp.Trap}\int_a^b {f\left( x \right)dx}  -
\frac{{h}}{2}\left[ {f\left( a \right) + 2 \sum\limits_{j = 1}^{n
- 1} {f\left( {x_j } \right)}  + f\left( b \right)}
\right]=\frac{{\left( {b - a} \right) }}{{12 }}h^2f''\left( {\mu}
\right),
\end{align}
for some $\mu \in (a,b)$, where $f \in C^2[a,b]$, $h =
\frac{b-a}{n}$ and $x_j = a+jh$, for each $j=0,1,\cdots,n$.

The main purpose of this work, is to combine the composite
Trapezoid and composite Midpoint formulae in an inequality that is
similar to the classical Hermite--Hadamard inequality
(\ref{eq1.3}) for convex functions defined on a real interval
$[a,b]$. In this way, we establish a conventional generalization
of (\ref{eq1.3}) which is in turn most useful and have a very
constructional form.

\section{A Generalization of Hermite--Hadamard's
Inequality}
\begin{theorem}
\label{thm1} Let $f:[a,b] \to \mathbb{R}$ be a convex function on
$[a,b]$, then the double inequality
\begin{align}
h\sum\limits_{k = 1}^n {f\left( {\frac{{x_{k - 1}  + x_k }}{2}}
\right)} \le \int_a^b {f\left( t \right)dt}\label{eqM} \le
\frac{h}{2}\left[ {f\left( a \right) + 2\sum\limits_{k = 1}^{n -
1} {f\left( {x_k } \right)}  + f\left( b \right)} \right],
\end{align}
holds, where $x_k = a + k \frac{b-a}{n}$, $k=0,1,2, \cdots, n$;
with $h = \frac{b-a}{n}$, $n \in \mathbb{N}$. The constant `$1$'
in the left-hand side and `$\frac{1}{2}$' in the right-hand side
are the best possible for all $n \in \mathbb{N}$. If $f$ is
concave then the inequality is reversed.
\end{theorem}

\begin{proof}
Since $f$ is convex on $[a,b]$, then $f$ so is on each subinterval
$[x_{j-1},x_{j}]$, $j=1,\cdots,n$, then for all $t \in [0,1]$, we
have
\begin{align}
\label{eq2.2}f\left( {tx_{j-1} + \left( {1 - t} \right)x_{j}}
\right) \le tf\left( x_{j-1} \right) + \left( {1 - t}
\right)f\left( {x_{j}} \right).
\end{align}
Integrating (\ref{eq2.2}) with respect to $t$ on $[0,1]$ we get
\begin{align}
\label{eq2.3}\int_0^1 {f\left( {tx_{j-1} + \left( {1 - t}
\right)x_{j}} \right)dt}  \le \frac{{f\left( x_{j-1} \right) +
f\left( {x_{j}} \right)}}{2}.
\end{align}
Substituting $u=tx_{j-1} + \left( {1 - t} \right)x_{j}$, in the
left hand side of (\ref{eq2.3}), we get
\begin{align*}
\int_{x_{j-1}}^{x_j} { f\left( {u} \right)du } \le
\frac{x_{j}-x_{j-1}}{2} \left(f\left( x_{j-1} \right) + f\left(
{x_{j}} \right)\right).
\end{align*}
Taking the sum over $j$ from $1$ to $n$, we get
\begin{align}
\label{eq2.4} \sum\limits_{j = 1}^{n} {\int_{x_{j-1}}^{x_j} {
f\left( {u} \right)du }}&=  \int_{a}^{b} { f\left( {u} \right)du }
\nonumber\\
&\le \sum\limits_{j = 1}^{n } {\frac{x_{j}-x_{j-1}}{2}
\left(f\left( x_{j-1} \right) + f\left( {x_{j}} \right)\right)}
\nonumber\\
&\le  \frac{1}{2}\mathop {\max }\limits_j \left\{ {x_{j }  -
x_{j-1} } \right\} \cdot \sum\limits_{j = 1}^{n} { \left(f\left(
x_{j-1} \right) + f\left( {x_{j}} \right)\right)}
\nonumber\\
&=\frac{h}{2} \left[ {f\left( {x_0 } \right)+f\left( {x_1 }
\right) + \sum\limits_{j = 2}^{n - 1} {\left\{ {f\left( {x_{j-1} }
\right) + f\left( {x_{j} } \right)} \right\}}  +f\left( {x_{n-1} }
\right)+ f\left( {x_n } \right)} \right]
\nonumber\\
&= \frac{h}{2}\left[ {f\left( {a} \right) + 2\sum\limits_{j =
1}^{n - 1} {f\left( {x_j } \right)}  + f\left( {b} \right)}
\right].
\end{align}
On the other hand, again since $f$ is convex on $I_x$, then for $t
\in [0,1]$, we have
\begin{align}
f\left( {\frac{{x_{j-1} + x_j}}{2}} \right) &= f\left(
{\frac{{tx_{j} + \left( {1 - t} \right)x_{j-1}}}{2} +
\frac{{\left( {1 - t} \right)x_{j} + tx_{j-1}}}{2}} \right)
\nonumber\\
&\le \frac{1}{2}\left[ {f\left( {tx_{j} + \left( {1 - t}
\right)x_{j-1}} \right) + f\left( {\left( {1 - t} \right)x_j +
tx_{j-1}} \right)} \right].\label{eq2.5}
\end{align}
Integrating inequality (\ref{eq2.5}) with respect to $t$ on
$[0,1]$ we get
\begin{align}
f\left( {\frac{{x_{j-1} + x_j}}{2}} \right) &\le
\frac{1}{2}\int_0^1 {\left[ {f\left( {tx_{j} + \left( {1 - t}
\right)x_{j-1}} \right) + f\left( {\left( {1 - t} \right)x_j +
tx_{j-1}} \right)} \right]dt}
\nonumber\\
&= \frac{1}{2}\int_0^1 {f\left( {tx_{j} + \left( {1 - t}
\right)x_{j-1}} \right)dt}  + \frac{1}{2}\int_0^1 {f\left( {\left(
{1 - t} \right)x_j + tx_{j-1}} \right)dt}.\label{eq2.6}
\end{align}
By putting $1 - t = s$ in the second integral on the right-hand
side of (\ref{eq2.6}), we have
\begin{align}
f\left( {\frac{{x_{j-1} + x_j}}{2}} \right) &\le
 \frac{1}{2}\int_0^1 {f\left( {tx_{j} + \left( {1 - t}
\right)x_{j-1}} \right)dt}  + \frac{1}{2}\int_0^1 {f\left( {\left(
{1 - t} \right)x_j + tx_{j-1}} \right)dt}
\nonumber\\
&=  \int_0^1 {f\left( {tx_{j} + \left( {1 - t} \right)x_{j-1}}
\right)dt}.\label{eq2.7}
\end{align}
Substituting $u=tx_{j} + \left( {1 - t} \right)x_{j-1}$, in the
left hand side of (\ref{eq2.7}), and then taking the sum over $j$
from $1$ to $n$, we get
\begin{align}
\label{eq2.8}h\sum\limits_{k = 1}^n {f\left( {\frac{{x_{k - 1}  +
x_k }}{2}} \right)} \le \int_a^b {f\left( t \right)dt}.
\end{align}
From (\ref{eq2.4}) and (\ref{eq2.8}), we get the desired
inequality (\ref{eqM}).

To prove the sharpness let (\ref{eqM}) hold with another constants
$C_1,C_2>0$, which gives
\begin{align}
C_1 \cdot h\sum\limits_{k = 1}^n {f\left( {\frac{{x_{k - 1}  + x_k
}}{2}} \right)} &\le \int_a^b {f\left( t \right)dt}
\nonumber\\
&\le C_2\cdot h\left[ {f\left( a \right) + 2\sum\limits_{k = 1}^{n
- 1} {f\left( {x_k } \right)}  + f\left( b \right)}
\right].\label{eq.Sharp}
\end{align}
Let $f:[a,b]\to \mathbb{R}$ be the identity map $f(x)=x$, then the
right-hand side of (\ref{eq.Sharp}) reduces to
\begin{align*}
\frac{b^2-a^2}{2}&\le C_2\cdot h\left[ {a + 2\sum\limits_{k =
1}^{n - 1} {x_k }  + b} \right]
\\
&=C_2\cdot h\left[ { a + 2\sum\limits_{k = 1}^{n - 1} { \left({a +
k\frac{{b - a}}{n}}\right)}  + b} \right]
\\
&=C_2 \cdot h\left[ {a + 2\sum\limits_{k = 1}^{n - 1} { a} +
2\frac{{b - a}}{n}\sum\limits_{k = 1}^{n - 1}{k}  + b} \right]
\\
&=C_2 \cdot \frac{b-a}{n}\left[ {a + 2\left( {n-1} \right)a +
2\frac{{b - a}}{n}\cdot \frac{n\left({n-1}\right)}{2}  + b}
\right]
\\
&=C_2 \cdot \left( {b-a} \right)\left( {a + b} \right).
\end{align*}
It follows that $\frac{1}{2}\le C_2$, i.e., $\frac{1}{2}$ is the
best possible constant in the right-hand side of (\ref{eqM}).

For the left-hand side, we have
\begin{align*}
\frac{b^2-a^2}{2} &\ge C_1 \cdot h\sum\limits_{k = 1}^n
{\frac{{x_{k - 1} + x_k }}{2}}
\\
&= C_1 \cdot h\sum\limits_{k = 1}^n {\left\{ {a + \left( {2k - 1}
\right)\frac{{b - a}}{{2n}}} \right\} }.
\\
&=C_1 \cdot  \frac{b-a}{n} \cdot \left[{na + \frac{{b - a}}{{2n}}
\left( {2 \cdot \frac{{n\left( {n + 1} \right)}}{2} - n}
\right)}\right]
\\
&= C_1 \cdot \frac{b^2-a^2}{2},
\end{align*}
which means that $1\ge C_1$, and thus $1$ is the best possible
constant in the left-hand side of (\ref{eqM}). Thus the proof of
Theorem \ref{thm1} is completely finished.
\end{proof}

\begin{remark}
In Theorem \ref{thm1}, if we take $n=1$, then we refer to the
original Hermite--Hadamard inequality (\ref{eq1.3}).
\end{remark}

As application, next we give a direct refinements of
Hermite-Hadamard's type inequalities for convex functions defined
on a real interval $[a,b]$, according to the number of division
`$n$' (e.g. $n=1,2,3,4$) in Theorem \ref{thm1}.
\begin{corollary}
In Theorem \ref{thm1}, we have
\begin{enumerate}
\item If $n=1$, then
\begin{align*}
 \left( {b-a} \right) f\left( {\frac{{a  + b }}{2}} \right) \le \int_a^b {f\left( t
\right)dt} \le \left( {b-a} \right)\frac{f\left( a \right)+
f\left( b \right)}{2}.
\end{align*}
\item If $n=2$, then
\begin{align*}
&\frac{\left( {b-a} \right)}{2}\left[{f\left( {\frac{{3a  + b
}}{4}} \right)+f\left( {\frac{{a + 3b }}{4}} \right)}\right]
\\
&\le \int_a^b {f\left( t \right)dt}
\\
&\le \frac{\left( {b-a} \right)}{4}\left[ {f\left( a \right) +
2f\left( {\frac{a+b}{2} } \right)  + f\left( b \right)} \right].
\end{align*}

\item If $n=3$, then
\begin{align*}
&\frac{{\left( {b - a} \right)}}{3}\left[ {f\left( {\frac{{5a +
b}}{6}} \right) + f\left( {\frac{{a + b}}{2}} \right) + f\left(
{\frac{{a + 5b}}{6}} \right)} \right]
\\
&\le \int_a^b {f\left( t \right)dt}
\\
&\le \frac{{\left( {b - a} \right)}}{6}\left[ {f\left( a \right) +
2f\left( {\frac{{2a + b}}{3}} \right) + 2f\left( {\frac{{a +
2b}}{3}} \right) + f\left( b \right)} \right].
\end{align*}

\item If $n=4$, then
\begin{align*}
&\frac{{\left( {b - a} \right)}}{4}\left[ {f\left( {\frac{{7a +
b}}{8}} \right) + f\left( {\frac{{5a + 3b}}{8}} \right) + f\left(
{\frac{{3a + 5b}}{8}} \right) + f\left( {\frac{{a + 7b}}{8}}
\right)} \right]
\\
&\le \int_a^b {f\left( t \right)dt}
\\
&\le \frac{{\left( {b - a} \right)}}{{12}}\left[ {f\left( a
\right) + 2f\left( {\frac{{3a + b}}{4}} \right) + 2f\left(
{\frac{{a + b}}{2}} \right) + 2f\left( {\frac{{a + 3b}}{4}}
\right) + f\left( b \right)} \right].\\
\end{align*}
\end{enumerate}
\end{corollary}

Now, let $f:[a.b] \to \mathbb{R}$ be a convex function on $[a,b]$.
Define the mappings $H_j,F_j:[0,1] \to \mathbb{R}$, given by
\begin{align}
\label{H1.map}H_j \left( {t} \right) = \frac{1}{h}\int_{x_{j-1}}^{
x_j} {f\left( {tu + \left( {1 - t} \right)\frac{{{x_{j-1}} +
x_j}}{2}} \right)du}, \,\,\,\,\,\,\,\,\,\,\, u \in [{x_{j-1}},{
x_j} ],
\end{align}
and
\begin{multline}
\label{F1.map} F_j\left( {t} \right) =
\frac{1}{{h}}\int_{x_{j-1}}^{x_j} \left[ {f\left( {\frac{{1 +
t}}{2}x_{j-1} + \frac{{1 - t}}{2}u} \right)+ f\left( {\frac{{1 +
t}}{2}x_j + \frac{{1 - t}}{2}u} \right)} \right]du, \,\,\,\, u \in
[{x_{j-1}},{ x_j} ].\\
\end{multline}

Applying Theorems \ref{dragomir.thm} and \ref{yang.thm}, for
$f:[x_{j-1},x_j]\to \mathbb{R}$, $j=1,1,\cdots,n$. Then the
following statements hold:
\begin{enumerate}
\item $H_j\left( {t} \right)$ and $F_j\left( {t} \right)$ are
convex for all $t \in [0,1]$ and $u \in [{x_{j-1}},{ x_j} ]$.

\item $H_j \left( {t} \right)$ and $F_j \left( {t} \right)$ are
monotonic nondecreasing for all $t \in [0,1]$ and $u \in
[{x_{j-1}},{ x_j} ]$.

\item We have the following bounds for $H_j \left( {t} \right)$
\begin{align}
\label{b1}  \frac{1}{h}\int_{x_{j-1}}^{ x_j}  {f\left( u
\right)du} = H_j\left( {1} \right),
\end{align}
and
\begin{align}
\label{b2} f\left( {\frac{{x_{k - 1}  + x_k }}{2}} \right)=
H_j\left( {0} \right).
\end{align}
and the following bounds for $F_j \left( {t} \right)$
\begin{align}
\label{b3} \frac{f\left( {x_{j-1} } \right)+f\left( {x_j }
\right)}{2} = F_j\left( {1} \right),
\end{align}
and
\begin{align}
\label{b4}\frac{1}{h}\int_{x_{j-1}}^{ x_j}  {f\left( u \right)du}
= F_j\left( {0} \right).
\end{align}
\end{enumerate}
Hence, we may establish two related mappings for the inequality
(\ref{eqM}).
\begin{proposition}
Let $f$ be as in Theorem \ref{thm1}, define the mappings
$H,F:[0,1] \to \mathbb{R}$, given by
\begin{align}
\label{H.map}H \left( {t} \right) =\sum\limits_{j = 1}^n {H_j
\left( t \right)}\,\,\,\,\,\emph{\text{and}}\,\,\,\,\,\,\,\,\,F
\left( {t} \right) =\sum\limits_{j = 1}^n {F_j \left( t \right)},
\end{align}
where $H_j \left( {t} \right)$ and $F_j\left( {t} \right) $ are
defined in (\ref{H1.map}) and (\ref{F1.map}), respectively;  then
the following statements hold:
\begin{enumerate}
\item $H\left( {t} \right)$ and $F\left( {t} \right)$ are convex
for all $t \in [0,1]$ and $u \in [a,b]$.

\item $H \left( {t} \right)$ and $F \left( {t} \right)$ are
monotonic nondecreasing for all $t \in [0,1]$ and $u \in [a,b]$.

\item We have the following bounds for $H \left( {t} \right)$
\begin{align}
\mathop {\sup }\limits_{t \in \left[ {0,1} \right]} H \left( {t}
\right)=  \frac{1}{h}\int_a^b {f\left( u \right)du}  = H\left( {1}
\right),
\end{align}
and
\begin{align}
\mathop {\inf }\limits_{t \in \left[ {0,1} \right]} H\left( {t}
\right) = \sum\limits_{k = 1}^n {f\left( {\frac{{x_{k - 1}  + x_k
}}{2}} \right)}  = H\left( {0} \right).
\end{align}
and the following bounds for $F \left( {t} \right)$
\begin{align}
\mathop {\sup }\limits_{t \in \left[ {0,1} \right]} F \left( {t}
\right)= \frac{1}{2}\left[ {f\left( a \right) + 2\sum\limits_{k =
1}^{n - 1} {f\left( {x_k } \right)}  + f\left( b \right)} \right]
= F\left( {1} \right),
\end{align}
and
\begin{align}
\mathop {\inf }\limits_{t \in \left[ {0,1} \right]} F \left( {t}
\right) =\frac{1}{h}\int_a^{b} {f\left( u \right)du} = F\left( {0}
\right).
\end{align}
\end{enumerate}
\end{proposition}

\begin{proof}
Taking the sum over $j$ from $1$ to $n$, in (\ref{b1})--(\ref{b4})
we get the required results, and we shall omit the details.
\end{proof}

\begin{remark}
The inequality (\ref{eqM}) may written in a convenient way as
follows:
\begin{align*}
\sum\limits_{k = 1}^n {f\left( {\frac{{x_{k - 1}  + x_k }}{2}}
\right)}- \sum\limits_{k = 1}^{n - 1} {f\left( {x_k } \right)} \le
\frac{1}{h}\int_a^b {f\left( t \right)dt} - \sum\limits_{k = 1}^{n
- 1} {f\left( {x_k } \right)}   \le \frac{f\left( a \right) +
f\left( b \right)}{2}
\end{align*}
which is of Ostrowski's type.
\end{remark}

Some sharps Ostrowski's type inequalities for convex functions
defined on a real interval $[a,b]$, are proposed in the next
theorems.
\begin{theorem}
\label{thm2} Let $f:[a,b] \to \mathbb{R}_+$ be a positive convex
function on $[a,b]$, then the inequality
\begin{align}
\int_{a}^{b} { f\left( {x} \right)dx } - \left( {b-a} \right)
f\left( {y} \right) \le \frac{h}{2}\left[ {f\left( a \right) +
2\sum\limits_{k = 1}^{n - 1} {f\left( {x_k } \right)}  + f\left( b
\right)} \right],\label{eq2.10}
\end{align}
holds for all $y \in [a,b]$. where, $x_k = a + k \frac{b-a}{n}$,
$k=0,1,2, \cdots, n$; with $h = \frac{b-a}{n}$, $n \in
\mathbb{N}$. The constant $\frac{1}{2}$ in the right-hand side is
the best possible, in the sense that it cannot be replaced by a
smaller one for all $n \in \mathbb{N}$. If $f$ is concave then the
inequality is reversed.
\end{theorem}

\begin{proof}
Fix $y \in [x_{j-1},x_{j}]$, $j=1,\cdots,n$. Since $f$ is convex
on $[a,b]$, then $f$ so is on each subinterval $[x_{j-1},x_{j}]$,
in particular on $[x_{j-1},y]$, then for all $t \in [0,1]$, we
have
\begin{align}
\label{eq2.11}f\left( {tx_{j-1} + \left( {1 - t} \right)y} \right)
\le tf\left( x_{j-1} \right) + \left( {1 - t} \right)f\left( {y}
\right),\,\,\,\,\,\,\,\,\,\,j=1,\cdots,n.
\end{align}
Integrating (\ref{eq2.11}) with respect to $t$ on $[0,1]$ we get
\begin{align}
\label{eq2.12}\int_0^1 {f\left( {tx_{j-1} + \left( {1 - t}
\right)y} \right)dt}  \le \frac{{f\left( x_{j-1} \right) + f\left(
{y} \right)}}{2}.
\end{align}
Substituting $u=tx_{j-1} + \left( {1 - t} \right)y$, in the left
hand side of (\ref{eq2.12}), we get
\begin{align}
\label{eq2.13}\int_{x_{j-1}}^{y} { f\left( {u} \right)du } \le
\frac{y-x_{j-1}}{2} \left(f\left( x_{j-1} \right) + f\left( {y}
\right)\right).
\end{align}
Now, we do similarly for the interval $[y,x_j]$, we therefore have
\begin{align}
\label{eq2.14}f\left( {ty + \left( {1 - t} \right)x_{j}} \right)
\le tf\left( y \right) + \left( {1 - t} \right)f\left( {x_{j}}
\right),\,\,\,\,\,\,\,\,\,\,j=1,\cdots,n.
\end{align}
Integrating (\ref{eq2.14}) with respect to $t$ on $[0,1]$ we get
\begin{align}
\label{eq2.15}\int_0^1 {f\left( {ty + \left( {1 - t} \right)x_{j}}
\right)dt}  \le \frac{{f\left( y \right) + f\left( {x_{j}}
\right)}}{2}.
\end{align}
Substituting $u=ty + \left( {1 - t} \right)x_{j}$, in the left
hand side of (\ref{eq2.15}), we get
\begin{align}
\label{eq2.16}\int_{y}^{x_{j}} { f\left( {u} \right)du } \le
\frac{x_{j}-y}{2} \left(f\left( y \right) + f\left( {x_{j}}
\right)\right).
\end{align}
Adding the inequalities (\ref{eq2.13}) and (\ref{eq2.16}), we get
\begin{align}
\int_{x_{j-1}}^{y} { f\left( {u} \right)du }+\int_{y}^{x_{j}} {
f\left( {u} \right)du } &=  \int_{x_{j-1}}^{x_j} { f\left( {u}
\right)du }
\nonumber\\
&\le \frac{y-x_{j-1}}{2} \left(f\left( x_{j-1} \right) + f\left(
{y} \right)\right) +\frac{x_{j}-y}{2} \left(f\left( y \right) +
f\left( {x_{j}} \right)\right)
\nonumber\\
&\le \frac{y-x_{j-1}}{2} \cdot f\left( x_{j-1} \right) +
\frac{x_{j}-y}{2}  \cdot f\left( {x_{j-1}} \right)
+\left({x_{j}-x_{j-1}}\right) f\left( {y} \right)\label{eq2.17}
\\
&\le\frac{x_j-x_{j-1}}{2} \left(f\left( x_{j-1} \right) + f\left(
{x_{j-1}} \right)\right) +h f\left( {y} \right).\nonumber
\end{align}
Taking the sum over $j$ from $1$ to $n$, we get
\begin{align*}
\sum\limits_{j = 1}^{n} {\int_{x_{j-1}}^{x_j} { f\left( {u}
\right)du }} &=  \int_{a}^{b} { f\left( {u} \right)du }
\nonumber\\
&=\sum\limits_{j = 1}^{n } {\frac{x_j-x_{j-1}}{2}\left\{{ f\left(
x_{j-1} \right) +  f\left( {x_{j}}\right)
}\right\}}+\sum\limits_{j = 1}^{n } {h f\left( {y} \right)}
\\
&\le  \frac{1}{2}\mathop {\max }\limits_j \left\{ {x_{j }  -
x_{j-1} } \right\} \cdot \sum\limits_{j = 1}^{n} { \left(f\left(
x_{j-1} \right) + f\left( {x_{j}} \right)\right)}+\left( {b - a}
\right) f\left( {y} \right)
\nonumber\\
&=\frac{h}{2} \left[ {f\left( {x_0 } \right)+f\left( {x_1 }
\right) + \sum\limits_{j = 2}^{n - 1} {\left\{ {f\left( {x_{j-1} }
\right) + f\left( {x_{j} } \right)} \right\}}  +f\left( {x_{n-1} }
\right)+ f\left( {x_n } \right)} \right]+\left( {b - a} \right)
f\left( {y} \right)
\nonumber\\
&= \frac{h}{2}\left[ {f\left( {a} \right) + 2\sum\limits_{j =
1}^{n - 1} {f\left( {x_j } \right)}  + f\left( {b} \right)}
\right]+\left( {b - a} \right) f\left( {y} \right),
\end{align*}
which gives that
\begin{align*}
\int_{a}^{b} { f\left( {u} \right)du } - \left( {b - a} \right)
f\left( {y} \right)\le \frac{h}{2}\left[ {f\left( {a} \right) +
2\sum\limits_{j = 1}^{n - 1} {f\left( {x_j } \right)}  + f\left(
{b} \right)} \right]
\end{align*}
for all $y \in [x_{j-1},x_j]\subseteq [a,b]$ for all
$j=1,2,\cdots,n$, which gives the desired result (\ref{eq2.10}).

To prove the sharpness let (\ref{eq2.10}) hold with another
constants $C>0$, which gives
\begin{align}
\int_{a}^{b} { f\left( {x} \right)dx } - \left( {b - a} \right)
f\left( {y} \right) \le C\cdot h\left[ {f\left( a \right) +
2\sum\limits_{k = 1}^{n - 1} {f\left( {x_k } \right)}  + f\left( b
\right)} \right].\label{eq2.18}
\end{align}
Let $f:[0,1]\to \mathbb{R}$ be the identity map $f(x)=x$, then the
right-hand side of (\ref{eq2.18}) reduces to
\begin{align*}
\frac{1}{2} -  y&\le C\cdot \frac{1}{n}\left[ { 2\sum\limits_{k =
1}^{n - 1} {x_k } + 1} \right]
\\
&=C \cdot \frac{1}{n}\left[ { 2\frac{{1}}{n}\cdot
\frac{n\left({n-1}\right)}{2}  + 1} \right]
\\
&=C.
\end{align*}
Choose $y=0$, it follows that $\frac{1}{2}\le C$, i.e.,
$\frac{1}{2}$ is the best possible constant in the right-hand side
of (\ref{eq2.10}).
\end{proof}

\begin{theorem}
\label{thm3} Under the assumptions of Theorem \ref{thm2}, we have
\begin{multline}
\int_{a}^{b} { f\left( {x} \right)dx } - \left( {b-a} \right)
f\left( {y} \right)
\\
\le \left[ {\frac{h}{2} + \mathop {\max }\limits_{1 \le j \le n}
\left| {y - \frac{{x_{j - 1}  + x_j }}{2}} \right|} \right]\cdot
\left[ {f\left( {a} \right) + 2\sum\limits_{j = 1}^{n - 1}
{f\left( {x_j } \right)}  + f\left( {b} \right)}
\right],\label{eq2.19}
\end{multline}
for all $y \in [a,b]$. The constant $\frac{1}{2}$ in the
right-hand side is the best possible for all $n \in \mathbb{N}$.
If $f$ is concave then the inequality is reversed.
\end{theorem}
\begin{proof}
Repeating the steps of the proof of Theorem \ref{thm2}, therefore
by (\ref{eq2.17})
\begin{align}
\int_{x_{j-1}}^{x_j} { f\left( {u} \right)du } &\le
\frac{y-x_{j-1}}{2} \cdot f\left( x_{j-1} \right) +
\frac{x_{j}-y}{2}  \cdot f\left( {x_{j-1}} \right)
\nonumber\\
&\qquad+h f\left( {y} \right)
\nonumber\\
&\le \max \left\{ {\frac{{y - x_{j - 1} }}{2},\frac{{x_j  -
y}}{2}} \right\}\cdot \left(f\left( x_{j-1} \right) + f\left(
{x_{j-1}} \right)\right)
\nonumber\\
&\qquad+h f\left( {y} \right)
\nonumber\\
&\le \left[ {\frac{{x_j  - x_{j - 1} }}{2} + \left| {y -
\frac{{x_{j - 1}  + x_j }}{2}} \right|} \right] \cdot
\left(f\left( x_{j-1} \right) + f\left( {x_{j-1}}
\right)\right)\label{eq2.20}
\\
&\qquad +h f\left( {y} \right)\nonumber
\end{align}
Taking the sum over $j$ from $1$ to $n$, we get
\begin{align*}
& \int_{a}^{b} { f\left( {u} \right)du }
\nonumber\\
&=\sum\limits_{j = 1}^{n } {\left[ {\frac{{x_j  - x_{j - 1} }}{2}
+ \left| {y - \frac{{x_{j - 1}  + x_j }}{2}} \right|}
\right]\cdot\left\{{ f\left( x_{j-1} \right) +  f\left(
{x_{j}}\right) }\right\}}+\sum\limits_{j = 1}^{n } {h f\left( {y}
\right)}
\\
&\le  \mathop {\max }\limits_{1 \le j \le n} \left[ {\frac{{x_j  -
x_{j - 1} }}{2} + \left| {y - \frac{{x_{j - 1}  + x_j }}{2}}
\right|} \right] \cdot \sum\limits_{j = 1}^{n} { \left(f\left(
x_{j-1} \right) + f\left( {x_{j}} \right)\right)}+ \left( {b - a}
\right)f\left( {y} \right)
\nonumber\\
&\le \left[ {\frac{h}{2} + \mathop {\max }\limits_{1 \le j \le n}
\left| {y - \frac{{x_{j - 1}  + x_j }}{2}} \right|} \right]
\\
&\qquad\times \left[ {f\left( {x_0 } \right)+f\left( {x_1 }
\right) + \sum\limits_{j = 2}^{n - 1} {\left\{ {f\left( {x_{j-1} }
\right) + f\left( {x_{j} } \right)} \right\}}  +f\left( {x_{n-1} }
\right)+ f\left( {x_n } \right)} \right]+\left( {b - a} \right)
f\left( {y} \right)
\nonumber\\
&= \left[ {\frac{h}{2} + \mathop {\max }\limits_{1 \le j \le n}
\left| {y - \frac{{x_{j - 1}  + x_j }}{2}} \right|}
\right]\cdot\left[ {f\left( {a} \right) + 2\sum\limits_{j = 1}^{n
- 1} {f\left( {x_j } \right)}  + f\left( {b} \right)}
\right]+\left( {b - a} \right) f\left( {y} \right),
\end{align*}
which gives that
\begin{multline*}
\int_{a}^{b} { f\left( {u} \right)du } - \left( {b - a} \right)
f\left( {y} \right)
\\
\le \left[ {\frac{h}{2} + \mathop {\max }\limits_{1 \le j \le n}
\left| {y - \frac{{x_{j - 1}  + x_j }}{2}} \right|} \right]\cdot
\left[ {f\left( {a} \right) + 2\sum\limits_{j = 1}^{n - 1}
{f\left( {x_j } \right)}  + f\left( {b} \right)} \right]
\end{multline*}
for all $y \in [x_{j-1},x_j]\subseteq [a,b]$ for all
$j=1,2,\cdots,n$, which gives the desired result (\ref{eq2.19}).
The proof of sharpness goes likewise the proof of the sharpness of
Theorem \ref{thm2} and we shall omit the details.
\end{proof}

\begin{corollary}
Let $\alpha_i\ge 0$, $\forall$ $i=0,1,2,\cdots,n$, be positive
real numbers such that $ \sum\limits_{i = 0}^n {\alpha _i }=1$,
then under the assumptions of Theorem \ref{thm3}, we have
\begin{multline}
\int_{a}^{b} { f\left( {x} \right)dx } - \left( {b-a} \right)
f\left( { \frac{1}{{n + 1}}\sum\limits_{i = 0}^n {\alpha _i x_i }
} \right)
\\
\le \left[ {\frac{h}{2} + \mathop {\max }\limits_{1 \le j \le n}
\left| { \frac{1}{{n + 1}}\sum\limits_{i = 0}^n {\alpha _i x_i } -
\frac{{x_{j - 1}  + x_j }}{2}} \right|} \right]\cdot \left[
{f\left( {a} \right) + 2\sum\limits_{j = 1}^{n - 1} {f\left( {x_j
} \right)}  + f\left( {b} \right)} \right],
\end{multline}
for all $y \in [a,b]$. The constant $\frac{1}{2}$ in the
right-hand side is the best possible. If $f$ is concave then the
inequality is reversed.
\end{corollary}

\begin{theorem}
\label{thm4} Under the assumptions of Theorem \ref{thm3}, we have
\begin{align}
\label{eq2.21}\frac{1}{b-a}\int_{a}^{b} { f\left( {u} \right)du }
- \frac{1}{n}\sum\limits_{j = 1}^{n } { f\left(
{\frac{x_{j-1}+x_j}{2}} \right)- \frac{1}{n}\sum\limits_{j = 1}^{n
- 1} {f\left( {x_j } \right)}} \le \frac{f\left( {a} \right) +   +
f\left( {b} \right)}{2n}
\end{align}
for all $j=1,2,\cdots,n$. The constant $\frac{1}{2}$ in the
right-hand side is the best possible. If $f$ is concave then the
inequality is reversed.
\end{theorem}
\begin{proof}
Repeating the steps of the proof of Theorem \ref{thm3},
(\ref{eq2.20}) if we choose $y=\frac{x_{j-1}+x_j}{2}$, then we get
\begin{align*}
\int_{x_{j-1}}^{x_j} { f\left( {u} \right)du } \le \frac{{x_j  -
x_{j - 1} }}{2}\cdot \left(f\left( x_{j-1} \right) + f\left(
{x_{j-1}} \right)\right) +h f\left( {\frac{x_{j-1}+x_j}{2}}
\right)
\end{align*}
Taking the sum over $j$ from $1$ to $n$, we get
\begin{align*}
& \int_{a}^{b} { f\left( {u} \right)du }
\nonumber\\
&\le  \sum\limits_{j = 1}^{n } {\frac{{x_j  - x_{j - 1}
}}{2}\cdot\left\{{ f\left( x_{j-1} \right) +  f\left(
{x_{j}}\right) }\right\}}+\sum\limits_{j = 1}^{n } {h f\left(
{\frac{x_{j-1}+x_j}{2}} \right)}
\nonumber\\
&\le  \frac{h}{2}\sum\limits_{j = 1}^{n } {\left\{{ f\left(
x_{j-1} \right) +  f\left( {x_{j}}\right) }\right\}}+ h
\sum\limits_{j = 1}^{n } { f\left( {\frac{x_{j-1}+x_j}{2}}
\right)},
\end{align*}
which gives that
\begin{align*}
\int_{a}^{b} { f\left( {u} \right)du } - h \sum\limits_{j = 1}^{n
} { f\left( {\frac{x_{j-1}+x_j}{2}} \right)} \le \frac{h}{2}
\left[ {f\left( {a} \right) + 2\sum\limits_{j = 1}^{n - 1}
{f\left( {x_j } \right)}  + f\left( {b} \right)} \right],
\end{align*}
for all $j=1,2,\cdots,n$, which gives the desired result
(\ref{eq2.21}). The proof of sharpness goes likewise the proof of
the sharpness of Theorem \ref{thm2} and we shall omit the details.
\end{proof}

\begin{theorem}
\label{thm5} Let $I \subset \mathbb{R}$ be an open interval and
$a, b \in I, a < b$. Let $f:I \to \mathbb{R}_+$ be an increasing
convex function on $[a,b]$, then the inequality
\begin{align}
\label{eq2.22}\int_a^b {f\left( t \right)dt}  - \frac{\left(b-a
\right)}{2}f\left( y \right) \ge \frac{h}{2}\sum\limits_{j = 1}^n
{f\left( {\frac{{x_{j - 1}  + x_j }}{2}} \right)} \ge 0,
\end{align}
is valid for all $y \in [a,b]\subset I$. The constant
$\frac{1}{2}$ in the right-hand side is the best possible, in the
sense that it cannot be replaced by a greater one. If $f$ is
concave then the inequality is reversed.
\end{theorem}

\begin{proof}
Let $y \in [x_{j-1},x_j]$ be an arbitrary point such that $x_{j -
1} < y - p \le y \le y + p < x_j$ for all $j=1,2,\cdots,n$ with
$p>0$.

It is well known that $f$ is convex on $I$ iff
\begin{align*}
f\left( y \right) \le \frac{1}{{2p}}\int_{y - p}^{y + p} {f\left(
t \right)dt}
\end{align*}
for every subinterval $[y-p,y+p]\subset [a,b] \subset I$ for some
$p>0$. But since $f$ increases on $[a,b]$, we also have
\begin{align*}
f\left( y \right) \le \frac{1}{{2p}}\int_{y - p}^{y + p} {f\left(
t \right)dt}  \le \frac{1}{{2p}}\int_{x_{j - 1} }^{x_j } {f\left(
t \right)dt}
\end{align*}
Choosing $p \ge \frac{h}{2}$, (this choice is available since it
is true for every subinterval in $I$), therefore from the last
inequality we get
\begin{align*}
f\left( y \right) \le \frac{1}{{h}}\int_{y - p}^{y + p} {f\left( t
\right)dt}  \le \frac{1}{{h}}\int_{x_{j - 1} }^{x_j } {f\left( t
\right)dt}.
\end{align*}
Again by convexity we have
\begin{align*}
hf\left( {\frac{{x_{j - 1}  + x_j }}{2}} \right) \le \int_{x_{j -
1} }^{x_j } {f\left( t \right)dt}
\end{align*}
Adding the last two inequalities, we get
\begin{align*}
h f\left( y \right) + hf\left( {\frac{{x_{j - 1}  + x_j }}{2}}
\right) \le 2\int_{x_{j - 1} }^{x_j } {f\left( t \right)dt}
\end{align*}
or we write
\begin{align*}
hf\left( {\frac{{x_{j - 1}  + x_j }}{2}} \right) \le 2\int_{x_{j -
1} }^{x_j } {f\left( t \right)dt}  - h f\left( y \right)
\end{align*}
Taking the sum over $j$ from $1$ to $n$, we get
\begin{align*}
h\sum\limits_{j = 1}^n {f\left( {\frac{{x_{j - 1}  + x_j }}{2}}
\right)}  \le 2\sum\limits_{j = 1}^n {\int_{x_{j - 1} }^{x_j }
{f\left( t \right)dt} }  - \sum\limits_{j = 1}^n {hf\left( y
\right)}
\end{align*}
hence,
\begin{align*}
\int_a^b {f\left( t \right)dt}  - \frac{\left(b-a
\right)}{2}f\left( y \right)\ge \frac{h}{2}\sum\limits_{j = 1}^n
{f\left( {\frac{{x_{j - 1}  + x_j }}{2}} \right)}  \ge 0,
\end{align*}
holds by positivity of $f$ and this proves our assertion.

To prove the sharpness let (\ref{eq2.22}) holds with another
constant $C>0$, which gives
\begin{align}
\int_a^b {f\left( t \right)dt}  - \frac{\left(b-a
\right)}{2}f\left( y \right) \ge C\cdot h \sum\limits_{j = 1}^n
{f\left( {\frac{{x_{j - 1}  + x_j }}{2}} \right)} \ge
0.\label{eq2.23}
\end{align}
Let $f:[0,1]\to \mathbb{R}_+$ be the identity map $f(x)=x$, then
the right-hand side of (\ref{eq2.23}) reduces to
\begin{align*}
\frac{1}{2} - \frac{1}{2} y&\ge C\cdot \frac{1}{n}\sum\limits_{j =
1}^{n}{\frac{2j-1}{2n}}
\\
&=C \cdot \frac{1}{n}\left[ {\sum\limits_{j =
1}^{n}{\frac{j}{n}}-\sum\limits_{j = 1}^{n}{\frac{1}{2n}}}\right]
\\
&=C \cdot \frac{1}{n}\left[ {\frac{1}{n}\cdot
\frac{n(n+1)}{2}-\frac{1}{2n}\cdot n}\right]
\\
&=\frac{1}{2}  C.
\end{align*}
Choose $y=\frac{1}{2}$, it follows that $\frac{1}{4}\ge
\frac{1}{2}C$ which means that $\frac{1}{2} \ge C$, i.e.,
$\frac{1}{2}$ is the best possible constant in the right-hand side
of (\ref{eq2.22}).
\end{proof}

\begin{corollary}
Let $\alpha_i\ge 0$, $\forall$ $i=0,1,2,\cdots,n$, be positive
real numbers such that $ \sum\limits_{i = 0}^n {\alpha _i }=1$,
then under the assumptions of Theorem \ref{thm5}, we have
\begin{align}
\int_a^b {f\left( t \right)dt}  - \frac{\left(b-a
\right)}{2}f\left(  \frac{1}{{n + 1}}\sum\limits_{i = 0}^n {\alpha
_i x_i } \right) \ge \frac{h}{2}\sum\limits_{j = 1}^n {f\left(
{\frac{{x_{j - 1}  + x_j }}{2}} \right)} \ge 0.
\end{align}
The constant $\frac{1}{2}$ in the right-hand side is the best
possible. If $f$ is concave then the inequality is reversed.
\end{corollary}

\centerline{}

\centerline{}

\end{document}